\documentclass[english,12pt,a4paper,twoside]{smfart}
\usepackage{amsfonts,smfthm}
\usepackage{amsmath}
\usepackage{stmaryrd}
\usepackage{amssymb}
\usepackage{graphics}
\usepackage{babel}
\usepackage{enumerate}
\usepackage{euscript}

\usepackage{newlfont}
\usepackage{latexsym}
\usepackage{graphicx}
\usepackage{a4wide}

\usepackage{cite}        
\usepackage{url}         

\title{Inverse zero-sum problems\\ in finite Abelian $p$-groups}
\author{Benjamin Girard}

\address{Centre de
Math\'{e}matiques Laurent Schwartz, UMR $7640$ du CNRS, \'Ecole
polytechnique, $91128$ Palaiseau cedex, France.}
\email{benjamin.girard@math.polytechnique.fr}

\theoremstyle{plain}
\newtheorem{theorem}{Theorem}[section]
\newtheorem{proposition}[theorem]{Proposition}
\newtheorem{lem}[theorem]{Lemma}
\newtheorem{corollary}[theorem]{Corollary}
\newtheorem{conjecture}[theorem]{Conjecture}

\theoremstyle{definition}

\theoremstyle{remark}

\def\cnp[#1,#2]{\begin{pmatrix} #1 \\#2 \end{pmatrix}}
\begin{document}

\maketitle \setcounter{page}{1} \vspace{0.0cm}

\begin{abstract} In this paper, we study the minimal number of elements of maximal order within a zero-sumfree sequence in a finite Abelian $p$-group. For this purpose, in the general context of finite Abelian groups, we introduce a new number, for which lower and upper bounds are proved in the case of finite Abelian $p$-groups. Among other consequences, the method that we use here enables us to show that, if we denote by $\exp(G)$ the exponent of the finite Abelian $p$-group $G$ which is considered, then a zero-sumfree sequence $S$ with maximal possible length in $G$ must contain at least $\exp(G)-1$ elements of maximal order, which improves a previous result of W.~Gao and A.~Geroldinger. 
\end{abstract}


\section{Introduction}
\label{Introduction}
Let $\mathcal{P}$ be the set of prime numbers and let $G$ be a finite Abelian group, written additively. By $\exp(G)$ we denote the exponent of $G$. If $G$ is cyclic of order $n$, it will be denoted by $C_n$. In the general case, we can decompose $G$ (see for instance \cite{Samuel}) as a direct product of cyclic groups $C_{n_1} \oplus \dots \oplus C_{n_r}$ where $1 < n_1 \text{ }
|\text{ } \dots \text{ }|\text{ } n_r \in \mathbb{N}$, so that every element $g$ of $G$ can be written $g=[a_1,\dots,a_r]$ (this notation will be used freely in this paper), with $a_i \in C_{n_i}$ for all $i \in \llbracket 1,r \rrbracket=\{1,\dots,r\}$.

\medskip
In this paper, any finite sequence $S=\left(g_1,\dots,g_\ell\right)$ of
$\ell$ elements from $G$, where repetitions are allowed and the ordering of the elements within $S$ is disregarded, will be called a \em sequence \em in $G$ with \em length \em $\left|S\right|=\ell$. For convenience, we will sometimes use the following notation, which is a shorter way to write a sequence $S$ when some of its elements appear several times. For every $g \in G$, we denote by $\mathsf{v}_g(S)$ the multiplicity of $g$ in $S$, so that:
$$S=\displaystyle\prod_{g \in G}g^{\mathsf{v}_g(S)}, \text{ where } \mathsf{v}_g(S) \in \mathbb{N} \text{ for all } g \in G.$$

\medskip
Given a sequence $S=\left(g_1,\dots,g_\ell\right)$ in $G$, we say that
$s \in G$ is a \em subsum \em of $S$ when
$$s=\displaystyle\sum_{i \in I} g_i \text{ for some }\emptyset \varsubsetneq I \subseteq \{1,\dots,\ell\}.$$

\noindent If $0$ is not a subsum of $S$, we say that $S$ is a \em
zero-sumfree sequence\em. If $\sum^\ell_{i=1} g_i=0,$ then $S$ is said
to be a \em zero-sum sequence\em. If moreover one has $\sum_{i \in
I} g_i \neq 0$ for all proper subsets $\emptyset \subsetneq I
\subsetneq \{1,\dots,\ell\}$, $S$ is called a \em minimal zero-sum
sequence\em.

\medskip
In a finite Abelian group $G$, the order of an element $g$ will be
written $\text{ord}_G(g)$ and for every divisor $d$ of the exponent of
$G$, we denote by $G_d$ the subgroup of $G$ consisting of all the
elements of order dividing $d$: $$G_d=\left\{x \in G \text{ } | \text{ } dx=0\right\}.$$ 

\medskip 
\noindent For every sequence $S$ of elements in $G$, we denote by $S_d$ the subsequence of $S$ consisting of all the
elements of order $d$ which are contained in $S$. 

\medskip 
Let $G \simeq C_{n_1} \oplus \dots \oplus C_{n_r},$ with $1 < n_1
\text{ } |\text{ } \dots \text{ }|\text{ } n_r \in \mathbb{N}$, be a finite Abelian 
group. We set:
$$\mathsf{D}^{*}(G)=\displaystyle\sum^r_{i=1} (n_i-1)+1 \hspace{0.2cm} \text{ as well as } \hspace{0.2cm} \mathsf{d}^{*}(G)=\mathsf{D}^{*}(G)-1.$$
By $\mathsf{D}(G)$ we denote the smallest integer $t \in \mathbb{N}^*$ such 
that every sequence $S$ in $G$ with $|S| \geq t$ contains a non-empty zero-sum subsequence.
The number $\mathsf{D}(G)$ is called the \em Davenport constant
\em of the group $G$. 

\medskip
\noindent By $\mathsf{d}(G)$ we denote the greatest length of a zero-sumfree sequence in $G$. It can be readily seen that for every finite Abelian group $G$, one has $\mathsf{d}(G)=\mathsf{D}(G)-1$. 

\medskip
If $G \simeq C_{\nu_1} \oplus \dots \oplus
C_{\nu_s}$, with $\nu_i > 1$ for all $i \in \llbracket 1,s
\rrbracket$, is the longest possible decomposition of $G$ into a
direct product of cyclic groups, then we set
$$\mathsf{k}^{*}(G)=\displaystyle\sum_{i=1}^s
\frac{\nu_i-1}{\nu_i}.$$ 

\medskip
\noindent The \em cross number \em of a sequence
$S=(g_1,\dots,g_\ell)$, denoted by $\mathsf{k}(S)$, is then defined by
$$\mathsf{k}(S)=\displaystyle\sum_{i=1}^{\ell} \frac{1}{\text{ord}_G(g_i)}.$$
The notion of cross number was introduced by U.~Krause in \cite{Krause84} (see also \cite{Krause91}). 
Finally, we define the so-called \em little cross number \em
$\mathsf{k}(G)$ of $G$:
$$\mathsf{k}(G)=\max\{\mathsf{k}(S) | S \text{ zero-sumfree sequence in } G \}.$$

\medskip 
Given a finite Abelian group $G$, two elementary constructions
(see \cite{GeroKoch05}, Proposition $5.1.8$) give
the following lower bounds:
$$\mathsf{d}^{*}(G) \leq \mathsf{d}(G) \hspace{0.5cm} \text{ and } \hspace{0.5cm} \mathsf{k}^{*}(G) \leq \mathsf{k}(G).$$

\medskip
The invariants $\mathsf{d}(G)$ and $\mathsf{k}(G)$ play a key r\^{o}le in the theory of non-unique factorization 
(see for instance Chapter $9$ in \cite{Narkie04}, the book \cite{GeroKoch05} which presents various aspects of the theory, and the survey \cite{GeroKoch06} also). They have been extensively studied during last decades and even if numerous results were proved (see Chapter $5$ of the book \cite{GeroKoch05}, \cite{GaoGero06} for a survey with many references on the subject, and \cite{Girard08} for recent results on the cross number of finite Abelian groups), their exact values are known for very special types of groups only. In the sequel, we will need some of these values in the case of finite Abelian $p$-groups, so we gather them into the following theorem (see \cite{Olso69a} and \cite{GeroCross94}). 

\begin{theorem} \label{Proprietes D et Cross number} Let $p \in \mathcal{P}$, $r \in \mathbb{N}^*$ and
$1 \leq a_1 \leq \dots \leq a_r$, where $a_i \in
\mathbb{N}$ for all $i \in \llbracket 1,r \rrbracket$. Then, for
the $p$-group $G \simeq C_{p^{a_1}} \oplus \cdots \oplus
C_{p^{a_r}}$, the following two statements hold.

\medskip
\begin{itemize}
\item[$(i)$] $$\mathsf{d}(G)=\displaystyle\sum^r_{i=1} \left(p^{a_i}-1\right)=\mathsf{d}^{*}(G).$$ 
\item[$(ii)$] $$\mathsf{k}(G)=\displaystyle\sum^r_{i=1} \left(\frac{p^{a_i}-1}{p^{a_i}}\right)=\mathsf{k}^{*}(G).$$
\end{itemize}
\end{theorem}

\medskip
In \cite{Olso69a}, J.~Olson actually proved a more general result than Theorem \ref{Proprietes D et Cross number} $(i)$, which will be useful in this article. So as to state this theorem, we need to introduce the following notation. For every element $g \in G$, the \em height \em of $g$, denoted by $\alpha(g)$, is defined in the following fashion:
$$\alpha(g)=\max\{p^n \text{ } | \text{ } \exists h \in G \text{ with } g=p^n h\}.$$

We can now state Olson's result. 
\begin{theorem}\label{Olson pour les p-groupes} Let $G$ be a finite Abelian $p$-group and $S=(g_1,\dots,g_\ell)$ be a sequence in $G$ such that one has:
$$\displaystyle\sum^\ell_{i=1} \alpha(g_i) > \mathsf{d}(G).$$
Then, $S$ cannot be a zero-sumfree sequence.
\end{theorem}

\section{Four inverse problems in zero-sum theory}
\label{Section conjectures}
Let $G$ be a finite Abelian group. What can be said about the exact structure of a long zero-sumfree sequence in $G$? The answer to this question, which would be useful in order to tackle problems in non-unique factorization theory, seems very difficult to obtain in general, and proves to highly rely on the structure of the group itself. Indeed, several results (see for instance \cite{GaoGero99}) show that one cannot hope to find a simple and exact structural characterization which would describe long zero-sumfree sequences in general. Nevertheless, one could try to find, instead of a complete characterization, some general properties which have to be satisfied by all the long zero-sumfree sequences, whatever the group is. In \cite{Girard09}, the author adressed two general conjectures concerning this type of inverse problems. 

\medskip
The first one bears upon the distribution of orders within a long zero-sumfree sequence in a finite Abelian group $G$, and is the following.
\begin{conjecture} \label{Conjecture Cross number} Let $G \simeq C_{n_1} \oplus \dots \oplus C_{n_r},$ with $1 < n_1
\text{ } |\text{ } \dots \text{ }|\text{ } n_r \in \mathbb{N}$, be a finite Abelian group.
Given a zero-sumfree sequence $S$ in $G$ verifying $\left| S \right| \geq \mathsf{d}^*(G)$, one always 
has the following inequality:
\begin{eqnarray*} \mathsf{k}(S) \leq \displaystyle\sum_{i=1}^r \left(\frac{n_i-1}{n_i}\right). \end{eqnarray*}
\end{conjecture}


The following theorem gathers what is currently known concerning Conjecture \ref{Conjecture Cross number}. Statements $(i)$, $(ii)$ and $(iii)$ were proved by the author in \cite{Girard09} (see Proposition $2.3$ and Theorem $2.4$). Statement $(iv)$ was recently proved by W.~Schmid in \cite{Schmid09} (see Corollary $3.4$).

\begin{theorem}\label{theorem Girard} Conjecture \ref{Conjecture Cross number} holds whenever:
\begin{itemize}
\item[$(i)$] $G$ is a finite Abelian $p$-group.

\item[$(ii)$] $G$ is a finite cyclic group.

\item[$(iii)$] $G$ is a finite Abelian group of rank two.

\item[$(iv)$] $G \simeq C_2 \oplus C_2 \oplus C_{2n}$, where $n \in \mathbb{N}^*$.
\end{itemize}
\end{theorem}

The reader interested in this type of problems is also referred to Section $7$ in \cite{Girard09}, where the following dual version of Conjecture \ref{Conjecture Cross number}, on the maximal possible length of a zero-sumfree sequence with large cross number, is discussed.

\begin{conjecture}\label{Conjecture Davenport} 
Let $G$ be a finite Abelian group and $G \simeq C_{\nu_1} \oplus \dots \oplus C_{\nu_s}$, with $\nu_i > 1$ for all $i \in \llbracket 1,s \rrbracket$, be its longest possible decomposition into a direct product of cyclic groups. 
Given a zero-sumfree sequence $S$ in $G$ verifying $\mathsf{k}(S) \geq \mathsf{k}^*(G)$, one always has the following inequality:
\begin{eqnarray*} \left| S \right| \leq \displaystyle\sum_{i=1}^s (\nu_i-1). \end{eqnarray*}
\end{conjecture}

It can readily be seen, using Theorem \ref{Proprietes D et Cross number} $(i)$, that Conjecture \ref{Conjecture Davenport} holds for finite Abelian $p$-groups, yet this conjecture remains widely open, even in the case of finite cyclic groups.



\medskip
In this article, we study two other inverse zero-sum problems. The first one deals with the minimal number of elements of maximal order within a long zero-sumfree sequence in a finite Abelian group. This question was raised and investigated by W.~Gao and A.~Geroldinger (see Section $6$ in \cite{GaoGero99}), and more recently, studied by the author in the case of finite Abelian groups of rank two (see Theorem $2.5$ in \cite{Girard09}). In the present paper, we consider the more general problem of the minimal number of elements of maximal order within any zero-sumfree sequence in a finite Abelian group, and obtain new results in the context of finite Abelian $p$-groups.




\medskip
In order to study this kind of inverse zero-sum problems, we propose to introduce the following number. Given a finite Abelian group $G$ and an integer $\delta \in \llbracket 0,\mathsf{d}(G)-1 \rrbracket$, we denote by $\Gamma_{\delta}(G)$ the minimal number of elements of maximal order contained in a zero-sumfree sequence $S$ with length $\left| S \right| \geq \mathsf{d}(G)-\delta$.

\medskip
In Section \ref{Outline of the method}, we present a general method which was introduced in \cite{Girard08} for the study of the cross number of finite Abelian groups. Then, using this method, we prove in Section \ref{Proof of the main theorem} the following theorem, which gives a lower bound for $\Gamma_{\delta}(G)$ in the special case of finite Abelian $p$-groups. 

\begin{theorem}\label{main theorem} Let $G \simeq C_{p^{a_1}} \oplus \cdots \oplus C_{p^{a_r}}$, where $p \in \mathcal{P}$, $r \in \mathbb{N}^*$ and $1 \leq a_1 \leq \dots \leq a_r$, with $a_i \in \mathbb{N}$ for all $i \in \llbracket 1,r \rrbracket$. Let also $\delta \in \llbracket 0,\mathsf{d}(G)-1 \rrbracket$ and $j_0=\min \{i \in \llbracket 1,r \rrbracket \text{ } | \text{ } a_i=a_r\}$. Then, one has:  
$$\Gamma_{\delta}(G) \geq (p^{a_r}-1)+(r-j_0)(p-1)p^{a_r-1}-\delta-\left\lfloor \frac{\delta}{\left(r-j_0+1\right)(p-1)} \right\rfloor.$$
\end{theorem}

This lower bound improves significantly a previous result of W.~Gao and A.~Geroldinger (see Corollary $5.1.13$ in \cite{GeroKoch05}), stating that every zero-sumfree sequence with maximal possible length in a finite Abelian $p$-group contains at least \em one \em element of maximal order. Indeed, by specifying $\delta=0$ in Theorem \ref{main theorem}, one obtains the following corollary.


\begin{corollary}\label{corollaire} Let $G$ be a finite Abelian $p$-group. Then, every zero-sumfree sequence $S$ in $G$ with $\left| S \right|=\mathsf{d}(G)$ contains at least $\exp(G)-1$ elements of maximal order.  
\end{corollary}

\medskip
In Section \ref{Proof of the main theorem} as well, we obtain a general upper bound for $\Gamma_{\delta}(G)$ in the case of finite Abelian $p$-groups (see Proposition \ref{majorant Gamma}), which, combined with the lower bound of Theorem \ref{main theorem}, implies the following result.

\begin{theorem}\label{main theorem 2} Let $p \in \mathcal{P}$, $r \in \mathbb{N}^*$ and
$1 \leq a_1 \leq \dots \leq a_{r-1} < a_r$, where $a_i \in \mathbb{N}$ for all $i \in \llbracket 1,r \rrbracket$. 
Then, for $G \simeq C_{p^{a_1}} \oplus \cdots \oplus C_{p^{a_r}}$ and $\delta \in \llbracket 0,\mathsf{d}(G)-1 \rrbracket$, we have: 
$$\Gamma_{\delta}(G)=\max\left(0,(p^{a_r}-1)-\delta-\left\lfloor \frac{\delta}{p-1} \right\rfloor\right).$$ 
\end{theorem}





In Section \ref{Proof of the main theorem bis}, we study the following general conjecture, which bears upon the greatest common divisor of the orders of the elements within a long zero-sumfree sequence in a finite Abelian group.

\begin{conjecture}\label{Conjecture ordre minimal d'un element d'une long zero-sumfree sequence}
Let $G \simeq C_{n_1} \oplus \dots \oplus C_{n_r},$ with $1 < n_1 \text{ } |\text{ } \dots \text{ }|\text{ } n_r \in \mathbb{N}$, be a finite Abelian group. Given a zero-sumfree sequence $S$ in $G$ verifying $\left| S \right| \geq \mathsf{d}^*(G),$ one has for all $g \in S$:
$$n_1 \text{ } | \text{ } \text{\em ord\em}_G(g).$$
\end{conjecture}


Conjecture \ref{Conjecture ordre minimal d'un element d'une long zero-sumfree sequence} is known to be true in the trivial case of finite cyclic groups. This conjecture also holds for finite Abelian groups of rank two (see Proposition $6.3.1$ in \cite{GaoGero99}), and we shall prove in Section \ref{Proof of the main theorem bis} that it holds for finite Abelian $p$-groups too, which is Statement $(i)$ in the following theorem. Statement $(iv)$ can be easily deduced from Theorem $3.1$ in \cite{Schmid09}.

\begin{theorem} \label{theorem recapitulatif} Conjecture \ref{Conjecture ordre minimal d'un element d'une long zero-sumfree sequence} holds whenever:
\begin{itemize}
\item[$(i)$] $G$ is a finite Abelian $p$-group.

\item[$(ii)$] $G$ is a finite cyclic group.

\item[$(iii)$] $G$ is a finite Abelian group of rank two.

\item[$(iv)$] $G \simeq C_2 \oplus C_2 \oplus C_{2n}$, where $n \in \mathbb{N}^*$.
\end{itemize}
\end{theorem} 

Finally, in Section \ref{concluding remark}, we propose and discuss one general conjecture concerning the behaviour of $\Gamma_{\delta}(G)$, when $G$ is a finite Abelian $p$-group.

\section{Outline of the method}
\label{Outline of the method}
Let $G$ be a finite Abelian group, and let $S$ be a sequence of elements in $G$. The general method that we will use in this paper (see also \cite{Girard08} and \cite{Girard09} for applications of this method in two other contexts), consists in considering, for every $d',d \in \mathbb{N}$ such that $1 \leq d' \text{ } | \text{ } d \text{ } | \text{ } \exp(G)$, the following exact sequence:
$$0 \rightarrow G_{d/d'} \hookrightarrow G_d \overset{\pi_{(d',d)}}\rightarrow \frac{G_d}{G_{d/d'}} \rightarrow 0.$$
Now, let $U$ be the subsequence of $S$ consisting of all the elements whose order divides $d$. If, for some $1 \leq d' \text{ } | \text{ } d \text{ } | \text{ } \exp(G)$, it is possible to find sufficiently many disjoint non-empty zero-sum subsequences in $\pi_{(d',d)}(U)$, that is to say sufficiently many disjoint subsequences in $U$ the sum of which are elements of order dividing $d/d'$, then $S$ cannot be a zero-sumfree sequence in $G$.

\medskip
So as to make this idea more precise, we proposed in \cite{Girard08} to introduce the following number, which can be seen as an extension of the classical Davenport constant. 

\medskip
Let $G \simeq C_{n_1} \oplus \dots \oplus C_{n_r},$ with $1 < n_1 \text{ } |\text{ } \dots \text{ }|\text{ } n_r \in \mathbb{N}$, be a finite Abelian group and $d',d \in \mathbb{N}$ be two integers such that $1 \leq d' \text{ } | \text{ } d \text{ } | \text{ } \exp(G)$. By $\mathsf{D}_{(d',d)}(G)$ we denote the smallest $t \in \mathbb{N}^*$ such that every sequence $S$ in $G_d$ with $|S| \geq t$ contains a non-empty subsequence of sum in $G_{d/d'}$. 

\medskip
Using this definition, we can prove the following simple lemma, which is one possible illustration of our idea. This result will be useful in Section \ref{Proof of the main theorem} and states that given a finite Abelian group $G$, there exist strong constraints on the way the orders of elements have to be distributed within a zero-sumfree sequence. 

\begin{lem}\label{lemme clef} Let $G$ be a finite Abelian group and $d',d \in \mathbb{N}$ be two integers such that $1 \leq d' \text{ } | \text{ } d \text{ } | \text{ } \exp(G)$. Given a sequence $S$ of elements in $G$, we will write $T$ for the subsequence of $S$ consisting of all the elements whose order divides $d/d'$, and we will write $U$ for the subsequence of $S$ consisting of all the elements whose order divides $d$ (In particular, one has $T \subseteq U$).  
Then, the following condition implies that $S$ cannot be a zero-sumfree sequence:
$$\left|T\right|+\left\lfloor\frac{\left|U\right|-\left|T\right|}{\mathsf{D}_{(d',d)}(G)}\right\rfloor \geq \mathsf{D}_{\left(\frac{d}{d'},\frac{d}{d'}\right)}(G).$$
\end{lem}

\begin{proof} Let us set $\mathsf{\Delta}=\mathsf{D}_{\left(\frac{d}{d'},\frac{d}{d'}\right)}(G)$. When it holds, this inequality implies that there are $\mathsf{\Delta}$ disjoint subsequences $S_1,\dots,S_{\mathsf{\Delta}}$ of $S$, the sum of which are elements of order dividing $d/d'$. Now, by the very definition of $\mathsf{D}_{\left(\frac{d}{d'},\frac{d}{d'}\right)}(G)$, $S$ has to contain a non-empty zero-sum subsequence. 
\end{proof}

\medskip
Now, in order to obtain effective inequalities from the symbolic constraints of Lemma \ref{lemme clef}, one can use a result proved in \cite{Girard08}, which states that for any finite Abelian group $G$ and every $1 \leq d' \text{ } | \text{ } d \text{ } | \text{ } \exp(G)$, the invariant $\mathsf{D}_{(d',d)}(G)$ is linked with the classical Davenport constant of a particular subgroup of $G$, which can be characterized explicitly. In order to define properly this particular subgroup, we have to introduce the following notation. 

\medskip
For all $i \in \llbracket 1,r \rrbracket$, we set:
$$A_i=\gcd(d',n_i), \text{ }
B_i=\frac{\mathrm{lcm}(d,n_i)}{\mathrm{lcm}(d',n_i)}$$

$$\text{ and } \text{} \upsilon_i(d',d)=\frac{A_i}{\gcd(A_i,B_i)}.$$

\medskip
For instance, whenever $d$ divides $n_i$, we have $\upsilon_i(d',d)=\gcd(d',n_i)=d'$, and in particular $\upsilon_r(d',d)=d'.$ We can now state our result on $\mathsf{D}_{(d',d)}(G)$ (see \cite{Girard08}, Proposition $3.1$).

\begin{proposition} \label{propmarrantegenerale} Let $G \simeq C_{n_1} \oplus \dots \oplus C_{n_r}$, with $1 < n_1
\text{ } |\text{ } \dots \text{ }|\text{ } n_r \in \mathbb{N}$, be a
finite Abelian group and $d',d \in \mathbb{N}$ be such that $1 \leq d' \text{ } | \text{ } d \text{ } | \text{ } \exp(G)$. Then, we have the following equality:
$$\mathsf{D}_{(d',d)}(G)=\mathsf{D}\left(C_{\upsilon_1(d',d)} \oplus
\dots \oplus C_{\upsilon_r(d',d)}\right).$$
\end{proposition}

\section{On the quantity $\Gamma_{\delta}(G)$ for finite Abelian $p$-groups }
\label{Proof of the main theorem}
In this section, we will show how the method presented in Section \ref{Outline of the method} can be used in order to study the minimal number of elements of maximal order within a zero-sumfree sequence in a finite Abelian $p$-group. First, we prove Theorem \ref{main theorem}, which, given a finite Abelian $p$-group $G$ and an integer $\delta \in \llbracket 0,\mathsf{d}(G)-1 \rrbracket$, gives a lower bound for the number $\Gamma_{\delta}(G)$.

\begin{proof}[Proof of Theorem \ref{main theorem}]
Let $S$ be a zero-sumfree sequence in $G \simeq C_{p^{a_1}} \oplus \cdots \oplus C_{p^{a_r}}$, with $\left|S\right| \geq \mathsf{d}(G)-\delta$. We set $d'=p$ and $d=p^{a_r}$, which leads to $d/d'=p^{a_r-1}$. Let also $T$ and $U$ be the two subsequences of $S$ which are defined in Lemma \ref{lemme clef}. In particular, one has $T \subseteq U=S$.

\medskip
To start with, we determine the exact value of $\mathsf{D}_{(d',d)}(G)$. One has, for every $i \in \llbracket 1,r \rrbracket$:
\begin{eqnarray*}
\upsilon_i(d',d) & = & \frac{p}{\gcd\left(p,\frac{p^{a_r}}{p^{a_i}}\right)}\\
                 & = & \begin{cases}
                       1 \text{ if } i < j_0,\\
                       p \text{ if } i \geq j_0. 
\end{cases}
\end{eqnarray*}
Therefore, using Proposition \ref{propmarrantegenerale} and Theorem \ref{Proprietes D et Cross number} $(i)$, we obtain:
\begin{eqnarray*}
\mathsf{D}_{(d',d)}(G) & = & \mathsf{D}\left(C_{\upsilon_1(d',d)} \oplus
\dots \oplus C_{\upsilon_r(d',d)}\right)\\
                       & = & \mathsf{D}(C^{r-j_0+1}_p)\\
                       & = & (r-j_0+1)(p-1)+1.
\end{eqnarray*}
Now, let us set, for all $i \in \llbracket 1,r \rrbracket$:
$$\beta_i=
\begin{cases}
a_i \hspace{0.75cm} \text{ if } i < j_0,\\
a_r-1 \text{ if } i \geq j_0. 
\end{cases}
$$
If we had the following inequality:
$$\left|T\right| > \displaystyle\sum^{r-1}_{i=1}(p^{\beta_i}-1)+\frac{\delta}{\left(r-j_0+1\right)(p-1)},$$
then it would imply that
\begin{eqnarray*}
\left|T\right|+\frac{\left|U\right|-\left|T\right|}{\mathsf{D}_{(d',d)}(G)} & \geq & \left|T\right|+\frac{\displaystyle\sum^{r}_{i=1}(p^{a_i}-1)-\delta-\left|T\right|}{(r-j_0+1)(p-1)+1}\\ 
& > & \displaystyle\sum^{r-1}_{i=1}(p^{\beta_i}-1)+\frac{\displaystyle\sum^{r}_{i=1}(p^{a_i}-1)-\displaystyle\sum^{r-1}_{i=1}(p^{\beta_i}-1)}{(r-j_0+1)(p-1)+1}\\
& = & \displaystyle\sum^{r-1}_{i=1}(p^{\beta_i}-1)+\frac{(p^{a_r}-1)+(r-j_0)(p^{a_r}-p^{a_r-1})}{(r-j_0+1)(p-1)+1}\\
& = & \displaystyle\sum^{r-1}_{i=1}(p^{\beta_i}-1)+\frac{\left((r-j_0+1)(p-1)+1\right)p^{a_r-1}-1}{(r-j_0+1)(p-1)+1}\\
& = & \displaystyle\sum^{r}_{i=1}(p^{\beta_i}-1)+1-\frac{1}{(r-j_0+1)(p-1)+1}\\
& = & \mathsf{D}_{\left(\frac{d}{d'},\frac{d}{d'}\right)}(G)-\frac{1}{\mathsf{D}_{(d',d)}(G)},
\end{eqnarray*}
and, according to Lemma \ref{lemme clef}, $S$ would contain a non-empty zero-sum subsequence, which is a contradiction.
Thus, one obtains
$$\left|T\right| \leq \displaystyle\sum^{r-1}_{i=1}(p^{\beta_i}-1)+\left\lfloor \frac{\delta}{\left(r-j_0+1\right)(p-1)} \right\rfloor,$$
which gives the following lower bound for the number of elements of maximal order contained in $S$: 
\begin{eqnarray*}
\left|S_{p^{a_r}}\right| & = & \left|S\right|-\left|T\right|\\
                         & \geq & \displaystyle\sum^r_{i=1}(p^{a_i}-1)-\delta-\displaystyle\sum^{r-1}_{i=1}(p^{\beta_i}-1) - \left\lfloor \frac{\delta}{\left(r-j_0+1\right)(p-1)} \right\rfloor\\
                         & = & (r-j_0+1)\left(p^{a_r}-1\right)-(r-j_0)\left(p^{a_r-1}-1\right)-\delta-\left\lfloor \frac{\delta}{\left(r-j_0+1\right)(p-1)} \right\rfloor\\
                         & = & (p^{a_r}-1)+(r-j_0)(p-1)p^{a_r-1}-\delta-\left\lfloor \frac{\delta}{\left(r-j_0+1\right)(p-1)}\right\rfloor,
\end{eqnarray*}
and the proof is complete.
\end{proof}

Given a finite Abelian $p$-group $G$ and an integer $\delta \in \llbracket 0,\mathsf{d}(G)-1 \rrbracket$, we can also obtain, using some explicit constructions, an upper bound for the number $\Gamma_{\delta}(G)$.

\begin{proposition}\label{majorant Gamma} Let $G \simeq C_{p^{a_1}} \oplus \cdots \oplus C_{p^{a_r}}$, with $p \in \mathcal{P}$, $r \in \mathbb{N}^*$ and $1 \leq a_1 \leq \dots \leq a_r$, with $a_i \in \mathbb{N}$ for all $i \in \llbracket 1,r \rrbracket$. Let also $\delta \in \llbracket 0,\mathsf{d}(G)-1 \rrbracket$ and $j_0=\min \{i \in \llbracket 1,r \rrbracket \text{ } | \text{ } a_i=a_r\}$. Then, one has 
$$\Gamma_{\delta}(G) \leq \max\left(0,(r-j_0+1)\left(p^{a_r}-1\right)-\delta-f(\delta)\right),$$
where
$$f(\delta)=\min\left(\left\lfloor \frac{\delta}{p-1} \right\rfloor,(r-j_0+1)\left(p^{a_r-1}-1\right)\right).$$ 
\end{proposition}

\begin{proof} Let $(e_1, \dots, e_r)$ be a basis of $G$, with $\text{ord}(e_i)=p^{a_i}$ for every $i \in \llbracket 1,r \rrbracket$. One can distinguish the three following cases.

\medskip 
$\textbf{Case 1.}$ If $0 \leq \delta < (r-j_0+1)(p-1)\left(p^{a_r-1}-1\right)$, then let us write
$$\delta=\delta_1(p-1)\left(p^{a_r-1}-1\right)+\delta_2, \text{ with } \delta_1 \in \llbracket 0,(r-j_0)\rrbracket \text{ and } \delta_2 \in \llbracket 0,(p-1)\left(p^{a_r-1}-1\right)-1 \rrbracket.$$
Thus, the sequence  
$$S=\left(\displaystyle\prod^{r-\delta_1-1}_{i=1}e^{p^{a_i}-1}_i\right)\left(\displaystyle\prod^{r-1}_{i=r-\delta_1}\left(e_i\right)^{p-1}\left(pe_i\right)^{p^{a_i-1}-1}\right) \left(e_{r}\right)^{p^{a_r}-1-\delta_2-\left\lfloor \frac{\delta_2}{p-1} \right\rfloor}\left(pe_{r}\right)^{\left\lfloor \frac{\delta_2}{p-1} \right\rfloor}$$
is a zero-sumfree sequence in $G$. On the one hand, since $\delta_1 \leq (r-j_0)$, one obtains
\begin{eqnarray*}
\left|S\right| & = & \displaystyle\sum^{r-\delta_1-1}_{i=1}\left(p^{a_i}-1\right)+\displaystyle\sum^{r-1}_{i=r-\delta_1}\left[(p-1)+\left(p^{a_r-1}-1\right)\right]+\left(p^{a_r}-1\right)-\delta_2-\left\lfloor \frac{\delta_2}{p-1} \right\rfloor+\left\lfloor \frac{\delta_2}{p-1} \right\rfloor\\
& = & \displaystyle\sum^{r}_{i=1}\left(p^{a_i}-1\right)+\displaystyle\sum^{r-1}_{i=r-\delta_1}\left[(p-1)+\left(p^{a_r-1}-1\right) - \left(p^{a_r}-1\right)\right]-\delta_2\\
& = & \displaystyle\sum^{r}_{i=1}\left(p^{a_i}-1\right)-\delta_1(p-1)\left(p^{a_r-1}-1\right) -\delta_2\\
& = & \mathsf{d}(G)-\delta. 
\end{eqnarray*}
On the other hand, $S$ contains the following number of elements of maximal order $p^{a_r}$:
\begin{eqnarray*}
\left|S_{p^{a_r}}\right| & = & \displaystyle\sum^{r-\delta_1-1}_{i=j_0}\left(p^{a_r}-1\right)+\displaystyle\sum^{r-1}_{i=r-\delta_1}(p-1)+\left(p^{a_r}-1\right)-\delta_2-\left\lfloor \frac{\delta_2}{p-1} \right\rfloor\\
                         & = & \left(r-\delta_1-j_0+1\right)\left(p^{a_r}-1\right)+\delta_1(p-1)-\delta_2-\left\lfloor \frac{\delta_2}{p-1} \right\rfloor\\                   
                         & = &                         \left(r-j_0+1\right)\left(p^{a_r}-1\right)-\delta-\delta_1\left(p^{a_r-1}-1\right)-\left\lfloor \frac{\delta_2}{p-1} \right\rfloor\\
                         & = & \left(r-j_0+1\right)\left(p^{a_r}-1\right)-\delta-\left\lfloor \frac{\delta}{p-1} \right\rfloor,
\end{eqnarray*} 
and we are done.

\medskip 
$\textbf{Case 2.}$ If $(r-j_0+1)(p-1)\left(p^{a_r-1}-1\right) \leq \delta < (r-j_0+1)(p-1)p^{a_r-1}$, then let us write
$$\delta'=\delta-(r-j_0+1)(p-1)\left(p^{a_r-1}-1\right),$$
and
$$\delta'=\delta'_1(p-1)+\delta'_2, \text{ with } \delta'_1 \in \llbracket 0,(r-j_0)\rrbracket \text{ and } \delta'_2 \in \llbracket 0,p-2 \rrbracket.$$
Thus, the sequence  
$$S=\left(\displaystyle\prod^{j_0-1}_{i=1}e^{p^{a_i}-1}_i\right)\left(\displaystyle\prod^{r-\delta'_1-1}_{i=j_0}\left(e_i\right)^{p-1}\left(pe_i\right)^{p^{a_r-1}-1}\right)\left(\displaystyle\prod^{r-1}_{i=r-\delta'_1}\left(pe_i\right)^{p^{a_r-1}-1}\right) \left(e_{r}\right)^{p-1-\delta'_2}\left(pe_{r}\right)^{p^{a_r-1}-1}$$
is a zero-sumfree sequence in $G$. On the one hand, since $\delta'_1 \leq (r-j_0)$, one obtains
\begin{eqnarray*}
\left|S\right| & = & \displaystyle\sum^{j_0-1}_{i=1}\left(p^{a_i}-1\right)+(r-\delta'_1-j_0)(p-1)+(r-j_0+1)\left(p^{a_r-1}-1\right)+(p-1)-\delta'_2\\
& = & \displaystyle\sum^{j_0-1}_{i=1}\left(p^{a_i}-1\right)+(r-j_0+1)(p-1)+(r-j_0+1)\left(p^{a_r-1}-1\right)-\delta'\\
& = & \displaystyle\sum^{j_0-1}_{i=1}\left(p^{a_i}-1\right)+(r-j_0+1)\left[(p-1)+\left(p^{a_r-1}-1\right)+(p-1)\left(p^{a_r-1}-1\right)\right]-\delta\\
& = & \displaystyle\sum^{j_0-1}_{i=1}\left(p^{a_i}-1\right)+(r-j_0+1)\left(p^{a_r}-1\right)-\delta\\
& = & \mathsf{d}(G)-\delta. 
\end{eqnarray*}
On the other hand, $S$ contains the following number of elements of maximal order $p^{a_r}$:
\begin{eqnarray*}
\left|S_{p^{a_r}}\right| & = & (r-\delta'_1-j_0)(p-1)+(p-1)-\delta'_2\\
                         & = & (r-\delta'_1-j_0)(p-1)+(p-1)-\delta'_2\\ 
                         & = & (r-j_0+1)(p-1)-\delta'\\                   
                         & = & (r-j_0+1)\left(p^{a_r}-1\right)-\delta-(r-j_0+1)\left(p^{a_r-1}-1\right),                       
\end{eqnarray*} 
and we are done. 

\medskip 
$\textbf{Case 3.}$ If $(r-j_0+1)(p-1)p^{a_r-1} \leq \delta \leq \mathsf{d}(G)-1$, then
\begin{eqnarray*}
(r-j_0+1)\left(p^{a_r}-1\right)-\delta-f(\delta) & \leq & (r-j_0+1)\left[\left(p^{a_r}-1\right)-(p-1)p^{a_r-1}-\left(p^{a_r-1}-1\right)\right]\\
& \leq & 0,
\end{eqnarray*}
as well as
\begin{eqnarray*}
\mathsf{d}(G)-\delta & \leq & \displaystyle\sum^r_{i=1}\left(p^{a_i}-1\right)-(r-j_0+1)(p-1)p^{a_r-1}\\
& = & \displaystyle\sum^{j_0-1}_{i=1}\left(p^{a_i}-1\right)+(r-j_0+1)\left[\left(p^{a_r}-1\right)-(p-1)p^{a_r-1}\right]\\
& = & \displaystyle\sum^{j_0-1}_{i=1}\left(p^{a_i}-1\right)+(r-j_0+1)\left(p^{a_r-1}-1\right).
\end{eqnarray*}
Now, let us consider the zero-sumfree sequence
$$S=\left(\displaystyle\prod^{j_0-1}_{i=1}e^{p^{a_i}-1}_i\right)\left(\displaystyle\prod^{r}_{i=j_0}\left(pe_i\right)^{p^{a_r-1}-1}\right),$$
which does not contain any element of maximal order. Thus, choosing any subsequence of $S$ with length $\mathsf{d}(G)-\delta$, we obtain that $\Gamma_{\delta}(G)=0$, which is the desired result.
\end{proof}

It is now easy, using Theorem \ref{main theorem} and Proposition \ref{majorant Gamma}, to derive Theorem \ref{main theorem 2}, which gives, in the case where $j_0=r$, the exact value of the number $\Gamma_{\delta}(G)$ for every integer $\delta \in \llbracket 0,\mathsf{d}(G)-1 \rrbracket$. 
\begin{proof}[Proof of Theorem \ref{main theorem 2}]
Since $j_0=r$, one obtains the following lower bound:
$$\Gamma_{\delta}(G) \geq (p^{a_r}-1)-\delta-\left\lfloor \frac{\delta }{p-1} \right\rfloor.$$
Consequently, one can distinguish three cases.

\medskip 
$\textbf{Case 1.}$ If $0 \leq \delta < (p-1)\left(p^{a_r-1}-1\right)$, then the upper bound given by Proposition \ref{majorant Gamma} implies that
$$\Gamma_{\delta}(G) = (p^{a_r}-1)-\delta-\left\lfloor \frac{\delta }{p-1} \right\rfloor.$$

\medskip 
$\textbf{Case 2.}$ If $(p-1)\left(p^{a_r-1}-1\right) \leq \delta < (p-1)p^{a_r-1}$, then the upper bound of Proposition \ref{majorant Gamma} implies that
$$\Gamma_{\delta}(G) \leq (p^{a_r}-1)-\delta-\left(p^{a_r-1}-1\right).$$
Now, since
$$\left(p^{a_r-1}-1\right) = \left\lfloor \frac{\delta }{p-1} \right\rfloor,$$
one obtains the desired equality:
$$\Gamma_{\delta}(G) = (p^{a_r}-1)-\delta-\left\lfloor \frac{\delta }{p-1} \right\rfloor.$$

\medskip 
$\textbf{Case 3.}$ If $(p-1)p^{a_r-1} \leq \delta \leq \mathsf{d}(G)-1$, then Proposition \ref{majorant Gamma} implies that $\Gamma_{\delta}(G)=0$, and the proof is complete.
\end{proof}

\section{Proof of Theorem \ref{main theorem bis}}
\label{Proof of the main theorem bis}
To start with, we prove the following lemma, which can be seen as a little more general version of Proposition $4.3$ in \cite{GaoGero02}. 
\begin{lem}\label{heights of a long zero-sumfree sequence} Let $G$ be a finite Abelian $p$-group and $S=(g_1,\dots,g_\ell)$ be a zero-sumfree sequence in $G$ with $\left|S\right| \geq \mathsf{d}(G)-p+2$. Then, every element of $S$ has height $1$.
\end{lem}

\begin{proof} Suppose that there exists an element in $S$, say $g_1$, verifying $\alpha(g_1)>1$. Then $\alpha(g_1) \geq p $, and setting $T=S \backslash g_1$, we deduce that
\begin{eqnarray*} \displaystyle\sum^\ell_{i=1} \alpha(g_i) & \geq & p + \left|T\right|\\
                                                           & \geq & p + \left(\mathsf{d}(G)-p+1\right)\\
                                                           &   =  & \mathsf{d}(G) + 1\\
                                                           &   >  & \mathsf{d}(G). 
\end{eqnarray*}
Thus, by Theorem \ref{Olson pour les p-groupes}, $S$ cannot be a zero-sumfree sequence, which is a contradiction.
\end{proof}

We can now prove Theorem \ref{theorem recapitulatif} $(i)$, as a simple corollary of the following stronger theorem.

\begin{theorem}\label{main theorem bis}
Let $G \simeq C_{n_1} \oplus \dots \oplus C_{n_r},$ with $1 < n_1 \text{ } |\text{ } \dots \text{ }|\text{ } n_r \in \mathbb{N}$, be a finite Abelian $p$-group. Given a zero-sumfree sequence $S$ in $G$ verifying $\left| S \right| \geq \mathsf{d}(G)-p+2,$ one has for all $g \in S$:
$$n_1 \text{ } | \text{ } \text{\em ord\em}_G(g).$$
\end{theorem}

\begin{proof}
The sequence $S$, with $\left| S \right| \geq \mathsf{d}(G)-p+2$, is a zero-sumfree sequence. Thus, by Lemma \ref{heights of a long zero-sumfree sequence}, every element of $S$ has height $1$. Let $g=[a_1,\dots,a_r]$ be an element of $S$. The equality $\alpha(g)=1$ implies that there exists $i_0 \in \llbracket 1,r \rrbracket$ such that $p$ does not divide $a_{i_0}$. Therefore, one has $\text{ord}_{C_{n_{i_0}}}(a_{i_0})=n_{i_0}$, and we obtain
\begin{eqnarray*}
\text{ord}_G(g) &   =  & \displaystyle\max_{i \in \llbracket 1,r \rrbracket} \text{ord}_{C_{n_i}}(a_i)\\
                & \geq & \text{ord}_{C_{n_{i_0}}}(a_{i_0})\\
                &   =  & n_{i_0}\\    
                & \geq & n_1,
\end{eqnarray*}
which completes the proof.
\end{proof}

\section{A concluding remark}
\label{concluding remark}
Let $G$ be a finite Abelian $p$-group of rank $r \geq 1$. It would be interesting to find the exact value of $\Gamma_{\delta}(G)$ for every integer $\delta \in \llbracket 0,\mathsf{d}(G)-1 \rrbracket$. Regarding this problem, we propose the following conjecture, supported by Theorem \ref{main theorem}, and which states that the upper bound given by Proposition \ref{majorant Gamma} is actually the right value for $\Gamma_{\delta}(G)$. 

\begin{conjecture}\label{conjecture p-groupes} Let $G \simeq C_{p^{a_1}} \oplus \cdots \oplus C_{p^{a_r}}$, where $p \in \mathcal{P}$, $r \in \mathbb{N}^*$ and $a_1 \leq \dots \leq a_r$, with $a_i \in \mathbb{N}^*$ for all $i \in \llbracket 1,r \rrbracket$. Let $\delta \in \llbracket 0,\mathsf{d}(G)-1 \rrbracket$ and $j_0=\min \{i \in \llbracket 1,r \rrbracket \text{ } | \text{ } a_i=a_r\}$. Then, one has
$$\Gamma_{\delta}(G) = \max\left(0,(r-j_0+1)\left(p^{a_r}-1\right)-\delta-f(\delta)\right),$$
where 
$$f(\delta)=\min\left(\left\lfloor \frac{\delta}{p-1} \right\rfloor,(r-j_0+1)\left(p^{a_r-1}-1\right)\right).$$ 
\end{conjecture}

By our Theorem \ref{main theorem 2}, this conjecture holds true in the case where $j_0=r$. One can also notice that in the special case where $j_0=1$, Theorem \ref{main theorem bis} implies that Conjecture \ref{conjecture p-groupes} holds for every $\delta \in \llbracket 0,p-2 \rrbracket$.

\section*{Acknowledgments}
I am grateful to my Ph.D. advisor Alain Plagne for his help during the preparation of this paper. I would like also to thank Alfred Geroldinger and Wolfgang Schmid for useful comments on a preliminary version of this article, as well as the Centre de Recerca Matem\`{a}tica in Barcelona, where this work was begun.



\end{document}